\documentclass[reqno]{amsart}
\usepackage{amssymb,setspace,tikz,xcolor,mathrsfs,listings,multicol}
\usepackage{rotating}
\usepackage[vcentermath]{youngtab}
\usepackage{hyperref}
\usetikzlibrary{arrows,matrix}
\tikzset{tab/.style={matrix of math nodes,column sep=-.35, row sep=-.35,text height=7pt,text width=7pt,align=center,inner sep=2,font=\footnotesize}}

\newcommand{\g}{\mathfrak{g}}
\newcommand{\HH}{\mathcal{H}}

\newcommand{\inner}[2]{\left\langle #1, #2 \right\rangle}
\newcommand{\iso}{\cong}

\renewcommand{\mid}{:}

\newcommand{\QQ}{\mathbf{Q}}
\newcommand{\RC}{\operatorname{RC}} 

\newcommand{\wt}{\mathrm{wt}}
\newcommand{\ZZ}{\mathbf{Z}}

\newcommand{\case}[1]{\vspace{6pt} \noindent \underline{#1}:}

\definecolor{darkred}{rgb}{0.7,0,0} 
\newcommand{\defn}[1]{{\color{darkred}\emph{#1}}} 

\usepackage{listings}
\lstdefinelanguage{Sage}[]{Python}
{morekeywords={False,True},sensitive=true}
\usepackage{xparse}

\makeatletter
\protected\def\specialmergetwolists{%
  \begingroup
  \@ifstar{\def\cnta{1}\@specialmergetwolists}
    {\def\cnta{0}\@specialmergetwolists}%
}
\def\@specialmergetwolists#1#2#3#4{%
  \def\tempa##1##2{%
    \edef##2{%
      \ifnum\cnta=\@ne\else\expandafter\@firstoftwo\fi
      \unexpanded\expandafter{##1}%
    }%
  }%
  \tempa{#2}\tempb\tempa{#3}\tempa
  \def\cnta{0}\def#4{}%
  \foreach \x in \tempb{%
    \xdef\cnta{\the\numexpr\cnta+1}%
    \gdef\cntb{0}%
    \foreach \y in \tempa{%
      \xdef\cntb{\the\numexpr\cntb+1}%
      \ifnum\cntb=\cnta\relax
        \xdef#4{#4\ifx#4\empty\else,\fi\x#1\y}%
        \breakforeach
      \fi
    }%
  }%
  \endgroup
}
\makeatother

\usepackage{xparse}
\DeclareDocumentCommand\rpp{ m m g }{
	\foreach \x [count=\s from 1] in {#1}{
	        {\ifnum\s=1
	                \draw (0,-\s)--(\x,-\s);
	                \fi}
	   \draw (0,-\s-1) to (\x,-\s-1);
	   \foreach \y in {0, ..., \x} {\draw (\y,-\s)--(\y,-\s-1);}
	}
	\specialmergetwolists{/}{#1}{#2}\ziplist
	\foreach \x/\y [count=\yi from 1] in \ziplist{
	    \node[anchor=west,font=\scriptsize] at (\x,-\yi - .5) {$\y$};
	}
	\IfValueT {#3}
	{\foreach \z [count=\zi from 1] in {#3} {\node[anchor=east,font=\scriptsize] at (0,-\zi - .5) {$\z$};}}
	{}
}

\theoremstyle{plain}
\newtheorem{thm}{Theorem}[section]
\newtheorem{lemma}[thm]{Lemma}

\newtheorem{prop}[thm]{Proposition}
\newtheorem{cor}[thm]{Corollary}
\theoremstyle{definition}
\newtheorem{dfn}[thm]{Definition}
\newtheorem{ex}[thm]{Example}
\newtheorem{remark}[thm]{Remark}
\numberwithin{equation}{section}
\numberwithin{figure}{section}
\numberwithin{table}{section}
\usepackage[colorinlistoftodos]{todonotes}

\begin{document}
\title{Rigged configurations and the $\ast$-involution}
\author{Ben Salisbury}
\address{Department of Mathematics, Central Michigan University, Mt.\ Pleasant, MI 48859}
\email{ben.salisbury@cmich.edu}
\urladdr{http://people.cst.cmich.edu/salis1bt/}
\thanks{B.S.\ was partially supported by CMU Early Career grant \#C62847}
\author{Travis Scrimshaw}
\address{Department of Mathematics, University of Minnesota, Minneapolis, MN 55455}
\email{tscrimsh@umn.edu}
\urladdr{http://math.umn.edu/~tscrimsh/}
\thanks{T.S.\ was partially supported by RTG grant NSF/DMS-1148634.}

\begin{abstract}
We give an explicit description of the $\ast$-involution on the rigged configuration model for $B(\infty)$.
\end{abstract}

\maketitle

\section{Introduction}

The $\ast$-involution, sometimes referred to as Kashiwara's involution, is an involution on the crystal $B(\infty)$ that is induced from a subtle involutive antiautomorphism of $U_q(\g)$.  
The importance of $\ast$ in the theory of crystal bases and their applications cannot be understated.  Here are just a few of its applications.
\begin{enumerate}
\item Saito~\cite{Sai94} used the involution during the proof that Lusztig's PBW basis has a crystal structure isomorphic to $B(\infty)$, provided that $\g$ is a finite-dimensional semisimple Lie algebra.
\item Kamnitzer and Tingley~\cite{KT09} generalize the definition of the crystal commutor of Henriques and Kamnitzer \cite{HK06} in terms of the $\ast$-involution. This leads to a proof by Savage~\cite{Sav09} that the category of crystals forms a coboundary category over any symmetrizable Kac-Moody algebra.
\item In affine type $A$, Jacon and Lecouvey~\cite{JL09} prove that $\ast$ coincides with the Zelevinsky involution~\cite{MW86,Zel80} on the set of simple modules for the affine Hecke algebra.
\end{enumerate}
Several combinatorial realizations of the $\ast$-involution are known in the literature.  For example, Lusztig~\cite{Lusztig93} gave a description of the behavior of $\ast$ on Lusztig's PBW basis in the finite types, Kamnitzer~\cite{Kamnitzer07} showed that $\ast$ acts on an MV polytope by negation, Kashiwara and Saito \cite{KS97} gave a description of $\ast$ in terms of quiver varieties~\cite{KS97}, and Jacon and Lecouvey~\cite{JL09} give a description of the involution in terms of the multisegment model.  Such model-specific calculations of the $*$-crystal operators are important as, {\it a priori}, the algorithm for computing the action of these operators is not efficient \cite[Thm. 2.2.1]{K93} (see also \cite[Prop. 8.1]{K95}).

In this paper, the authors continue their development of the rigged configuration model of $B(\infty)$~\cite{SalS15,SalS15II}.  Rigged configurations are sequences of partitions, one for every node of the underlying Dynkin diagram, where each part is paired with an integer, satisfying certain conditions.  These objects arose as an important tool in mathematical physics from the studies of the Bethe Ansatz by Kerov, Kirillov, and Reshetikhin~\cite{KKR86,KR86}, and they have been shown to correspond to the action and angle variables of box-ball systems~\cite{KOSTY06}. Additionally, rigged configurations have been used extensively in the theory of Kirillov-Reshetikhin crystals \cite{HKOTT02,HKOTY99,OSS13,OSS03,OSS03II,Sakamoto14,S05,S06,SchillingS15,SW10,Scrimshaw15}. During the course of this study, a crystal structure was given to rigged configurations~\cite{S06, SchillingS15}.

Our description of $\ast$ is as nice as one could hope: in contrast to the definition of $e_a$ and $f_a$ on rigged configurations, one interchanges ``label'' and ``colabel'' to obtain a definition of $e_a^*$ and $f_a^*$ (see Definition \ref{def:RC_star_crystal_ops}).  In turn, applying $\ast$ to a rigged configuration replaces all labels with its corresponding colabels and leaves the partitions fixed (see Corollary~\ref{cor:RC_star_involution}).

The method of proof applied here is to use a classification theorem of $B(\infty)$ asserted by Tingley and Webster~\cite{TW} by translating the $\ast$-involution directly into the classification theorem of Kashiwara and Saito~\cite{KS97} without the use of Kashiwara's embedding.  This classification theorem requires several assertions to be satisfied, and proving these assertions hold in $\RC(\infty)$ with our new $\ast$-crystal operators consumes most of Section \ref{sec:*crystal}.

The (conjectural) bijection $\Phi$ between $U_q'(\g)$-rigged configurations and tensor products of Kirillov-Reshetikhin crystals~\cite{OSS13, OS12, OSS03, OSS03II, OSS03III, S05, SchillingS15, SS2006, Scrimshaw15} is given roughly as follows. It removes the largest row with a colabel of 0, which is the minimal colabel, for each $e_a$ from $b$ to the highest weight element in $B(\Lambda_1)$, where $b$ is the leftmost factor in the tensor product.  
Let $\theta$ be the involution on $U_q'(\g)$-rigged configurations which interchanges labels with colabels on classically highest weight $U_q'(\g)$-rigged configurations and let $\widetilde{\ast}^L$ denote the involution which is the composition of Lusztig's involution and the map sending the result to the classically highest weight element~\cite{S05, SS2006} (where it is also denoted by $\ast$). It is known that $\Phi \circ \theta = \widetilde{\ast}^L \circ \Phi$ on classically highest weight elements. In particular, the latter map reverses the order of the tensor product. Thus, given the description of the crystal commutor, our work suggests there is a strong link between the $\ast$-involution and the bijection $\Phi$. We hope this could lead to a more direct description of the bijection $\Phi$, its related properties, and a (combinatorial) proof of the $X = M$ conjecture of~\cite{HKOTT02, HKOTY99}.

Another model for $B(\infty)$ uses marginally large tableaux, as developed by Hong and Lee \cite{HL08, HL12}.  It is known that the bijection $\Phi$ mentioned above can be extended to a $U_q(\g)$-crystal isomorphism between rigged configurations and marginally large tableaux~\cite{SalS15III} when $\g$ is of finite classical type or type $G_2$.  An ambitious hope of this paper is that it may lead to a description of the $\ast$-crystal structure on marginally large tableaux.  (In finite type $A$, this result is in~\cite{CT15}.) However, this appears to be a hard problem as the bijection $\Phi$ is highly recursive and depends on conditions on colabels, many of which can change under applying the $\ast$-crystal operators.

There is also a model for $B(\infty)$ using Littelmann paths constructed by Li and Zhang~\cite{LZ11}.  From~\cite{PS15}, natural virtualization maps arise to the embeddings on the underlying geometric information. The virtualization map on rigged configurations is also quite natural, giving evidence that rigged configurations encode more geometry than their combinatorial origins and description suggests. This is also evidence that there exists a straightforward and natural explicit combinatorial bijection between rigged configurations and the Littelmann path model. Thus this work could potentially lead to a description of the $\ast$-crystal on the Littelmann path model.

In a similar vein, the virtualization map is known to act naturally on MV polytopes~\cite{JS15,NS08}, also reflecting the geometric information of the root systems via the Weyl group. This is evidence that there should be a natural explicit combinatorial bijection between MV polytopes and rigged configurations (and the Littelmann path model). Moreover, considering the $\ast$-involution, which acts by negation on MV polytopes~\cite{Kamnitzer07, Kamnitzer10}, this work gives further evidence that such a bijection should exist. Furthermore, this bijection would suggest a natural generalization beyond finite type, which the authors expect to recover the KLR polytopes of~\cite{TW}.

This paper is organized as follows.
In Section~\ref{sec:background}, we give the necessary background on crystals and the $\ast$-involution.
In Section~\ref{sec:RC}, we give background information on the rigged configuration model for $B(\infty)$.
In Section~\ref{sec:*crystal}, we give the proof of our main theorem and some consequences.
In Section~\ref{sec:hw_crystals}, we give a description of highest weight crystals using the $\ast$-crystal structure and describe the natural projection from $B(\infty)$ in terms of rigged configurations.

\section{Crystals and the $\ast$-involution}
\label{sec:background}

Let $\g$ be a symmetrizable Kac-Moody algebra with quantized universal enveloping algebra $U_q(\g)$ over $\QQ(q)$, index set $I$, generalized Cartan matrix $A = (A_{ij})_{i,j\in I}$, weight lattice $P$, root lattice $Q$, fundamental weights $\{\Lambda_i \mid i \in I\}$, simple roots $\{\alpha_i \mid i\in I\}$, and simple coroots $\{h_i \mid i\in I\}$.  There is a canonical pairing $\langle\ ,\ \rangle\colon P^\vee \times P \longrightarrow \ZZ$ defined by $\langle h_i, \alpha_j \rangle = A_{ij}$, where $P^{\vee}$ is the dual weight lattice.

An \defn{abstract $U_q(\g)$-crystal} is a set $B$ together with maps
\[
 e_i, f_i \colon B \longrightarrow B\sqcup\{0\},\qquad
\varepsilon_i,\varphi_i\colon B \longrightarrow \ZZ \sqcup \{-\infty\},\qquad
\wt\colon B \longrightarrow P
\]
satisfying certain conditions (see \cite{HK02,K95}).  Any $U_q(\g)$-crystal basis, defined in the classical sense (see \cite{K91}), is an abstract $U_q(\g)$-crystal.  In particular, the negative half $U_q^-(\g)$ of the quantized universal enveloping algebra of $\g$ has a crystal basis which is an abstract $U_q(\g)$-crystal.  We denote this crystal by $B(\infty)$ (rather than the using the entire tuple $(B(\infty),e_i,f_i,\varepsilon_i,\varphi_i,\wt)$), and denote its highest weight element by $u_\infty$.  As a set, one has
\[
B(\infty) = \{  f_{i_d} \cdots  f_{i_2} f_{i_1} u_\infty : i_1,\dots,i_d \in I, \ d \ge 0 \}.
\]
The remaining crystal structure on $B(\infty)$ is
\begin{align*}
\wt( f_{i_d} \cdots  f_{i_2} f_{i_1} u_\infty) &= -\alpha_{i_1}-\alpha_{i_2}-\cdots-\alpha_{i_d} ,\\
\varepsilon_i(b) &= \max\{ k \in \ZZ :  e_ib \neq 0 \}, \\
\varphi_i(b) &= \varepsilon_i(b) + \langle h_i,\wt(b) \rangle. 
\end{align*}
We say that $b\in B(\infty)$ has \defn{depth} $d$ if $b = f_{i_d} \cdots  f_{i_2} f_{i_1} u_\infty$ for some $i_1,\dots,i_d \in I$.

There is a $\QQ(q)$-antiautomorphism $*\colon U_q(\g) \longrightarrow U_q(\g)$ defined by
\[
E_i \mapsto E_i, \qquad
F_i \mapsto F_i, \qquad
q \mapsto q, \qquad
q^h \mapsto q^{-h}.
\]
This is an involution which leaves $U_q^-(\g)$ stable.  Thus, the map $\ast$ induces a map on $B(\infty)$, which we also denote by $*$, and is called the \defn{$\ast$-involution} or \defn{star involution} (and is sometimes known as Kashiwara's involution).  Denote by $B(\infty)^*$ the image of $B(\infty)$ under $*$. 

\begin{thm}[\cite{K93,Lusztig90}]
We have
$
B(\infty)^* = B(\infty).
$
\end{thm}

This induces a new crystal structure on $B(\infty)$ with Kashiwara operators
\[
 e_i^* = * \circ  e_i \circ *,\ \ \ \ \ \ 
 f_i^* = * \circ  f_i \circ *,
\]
and the remaining crystal structure is given by
\[
\varepsilon_i^* = \varepsilon_i \circ * , \qquad \qquad
\varphi_i^* = \varphi_i \circ *,
\]
and weight function $\wt$, the usual weight function on $B(\infty)$.
Additionally, for $b\in B(\infty)$ and $i\in I$, define
\begin{equation}
\label{eq:jump}
\kappa_i(b) := \varepsilon_i(b) + \varepsilon_i^*(b) + \langle h_i, \wt(b)\rangle.
\end{equation}
This was called the {\it $i$-jump} in \cite{LV11}.



We will appeal to the following statement from \cite{CT15}, which was proven in a dual form in \cite{TW} based on Kashiwara and Saito's classification theorem for $B(\infty)$ from \cite{KS97}.  First, a \defn{bicrystal} is a set $B$ with two abstract $U_q(\g)$-crystal structures $(B,e_i,f_i,\varepsilon_i,\varphi_i,\wt)$ and $(B,e_i^\star,f_i^\star,\varepsilon_i^\star,\varphi_i^\star,\wt)$ with the same weight function.  In such a bicrystal $B$, we say $b\in B$ is a \defn{highest weight element} if $e_ib = e_i^\star b = 0$ for all $i \in I$.

\begin{prop}
\label{prop:star_properties}
Fix a bicrystal $B$ with highest weight $b_0$ such the crystal data is determined by setting $\wt(b_0) = 0$.  Assume further that, for all $i \neq j$ in $I$ and all $b\in B$, 
\begin{enumerate}
\item\label{item:star1} $f_ib$, $f_i^\star b \neq 0$;
\item\label{item:star2} $f_i^\star f_jb = f_jf_i^\star b$;
\item\label{item:star3} $\kappa_i(b) \ge 0$;
\item\label{item:star4} $\kappa_i(b) = 0$ implies $f_ib = f_i^\star b$;
\item\label{item:star5} $\kappa_i(b) \ge 1$ implies $\varepsilon_i^\star(f_ib) = \varepsilon_i^\star(b)$ and $\varepsilon_i(f_i^\star b) = \varepsilon_i(b)$;
\item\label{item:star6} $\kappa_i(b) \ge 2$ implies $f_if_i^\star b = f_i^\star f_ib$.
\end{enumerate}
Then 
\[
(B,e_i,f_i,\varepsilon_i,\varphi_i,\wt) \iso (B,e_i^\star,f_i^\star,\varepsilon_i^\star,\varphi_i^\star,\wt) \iso B(\infty),\]
with $e_i^\star = e_i^*$ and $f_i^\star = f_i^*$.
\end{prop}

However, we will need to slightly weaken the assumptions of Proposition~\ref{prop:star_properties}.

\begin{prop}
\label{prop:weaker_conditions}
Let $(B,e_i,f_i,\varepsilon_i,\varphi_i,\wt)$ and $(B^\star,e_i^\star,f_i^\star,\varepsilon_i^\star,\varphi_i^\star,\wt)$ be highest weight abstract $U_q(\g)$-crystals with the same highest weight vector $b_0 \in B \cap B^\star$, where the remaining crystal data is determined by setting $\wt(b_0) = 0$. Suppose also that (\ref{item:star1})--(\ref{item:star6}) are satisfied. Then
\[
(B,e_i,f_i,\varepsilon_i,\varphi_i,\wt) \iso (B^\star,e_i^\star,f_i^\star,\varepsilon_i^\star,\varphi_i^\star,\wt) \iso B(\infty),
\]
with $e_i^\star = e_i^*$ and $f_i^\star = f_i^*$.
\end{prop}

\begin{proof}
We prove that $B \cap B^\star$ is closed under $f_i$ and $f_i^\star$ by using induction on the depth and making repeated use of conditions (\ref{item:star1})--(\ref{item:star6}) above. The base case is depth $0$, where we just have $b_0$. Suppose all elements of depth at most $d$ in $B$ and $B^\star$ are in $B \cap B^\star$. Next, fix some $b \in B \cap B^\star$ at depth $d$. If $\kappa_i(b) = 0$, then $f_i b = f_i^\star b$ for all $i\in I$. If $i \neq j$, then $f_i^\star f_j b' = f_j f_i^\star b'$. Hence $f_i^* b \in B \cap B^\star$, and $f_j b \in B \cap B^\star$ by our induction assumption and that $B$ (resp., $B^\star$) is closed under $f_j$ (resp., $f_i^\star$). A similar argument shows that $f_i b \in B \cap B^\star$ if $\kappa_i(b) \geq 1$ since $\kappa_i(b') \geq 2$. Therefore $B = B \cap B^\star = B^\star$ since $B \cap B^\star$ is closed under $f_i$ and $f_j^\star$ and generated by $b_0$ (along with $B$ and $B^\star$). Thus the claim follows by Proposition~\ref{prop:star_properties}.
\end{proof}

\section{Rigged configurations}
\label{sec:RC}

Let $\HH = I \times \ZZ_{>0}$. A rigged configuration is a sequence of partitions $\nu = (\nu^{(a)} \mid a \in I)$ such that each row $\nu_i^{(a)}$ has an integer called a \defn{rigging}, and we let $J = \bigl(J_i^{(a)} \mid (a, i) \in \HH \bigr)$, where $J_i^{(a)}$ is the multiset of riggings of rows of length $i$ in $\nu^{(a)}$. We consider there to be an infinite number of rows of length $0$ with rigging $0$; i.e., $J_0^{(a)} = \{0, 0, \dotsc\}$ for all $a \in I$. The term rigging will be interchanged freely with the term \defn{label}.  We identify two rigged configurations $(\nu, J)$ and $(\widetilde{\nu}, \widetilde{J})$ if 
\[
J_i^{(a)} = \widetilde{J}_i^{(a)}
\]
for any fixed $(a, i) \in \HH$. Let $(\nu, J)^{(a)}$ denote the rigged partition $(\nu^{(a)}, J^{(a)})$.

Define the \defn{vacancy numbers} of $\nu$ to be 
\begin{equation}
\label{eq:vacancy}
p_i^{(a)}(\nu) = p_i^{(a)} = - \sum_{(b,j) \in \HH} A_{ab} \min(i, j) m_j^{(b)},
\end{equation}
where $m_i^{(a)}$ is the number of parts of length $i$ in $\nu^{(a)}$.  The \defn{corigging}, or \defn{colabel}, of a row in $(\nu,J)^{(a)}$ with rigging $x$ is $p_i^{(a)} - x$.  In addition, we can extend the vacancy numbers to
\[
p_{\infty}^{(a)} = \lim_{i\to\infty} p_i^{(a)} =  - \sum_{b \in I} A_{ab} \lvert \nu^{(b)} \rvert
\]
since $\sum_{j=1}^{\infty} \min(i,j) m_j^{(b)} = \lvert \nu^{(b)} \rvert$ for $i \gg 1$.  Note this is consistent with letting $i = \infty$ in Equation~\eqref{eq:vacancy}.

Let $\RC(\infty)$ denote the set of rigged configurations generated by $(\nu_{\emptyset}, J_{\emptyset})$, where $\nu_{\emptyset}^{(a)} = 0$ for all $a \in I$, and closed under the crystal operators as follows.

\begin{dfn}
\label{def:RC_crystal_ops}
Fix some $a \in I$, and let $x$ be the smallest rigging in $(\nu,J)^{(a)}$.
\begin{itemize}
\item[\defn{$e_a$}:] If $x =0$, then $e_a(\nu, J) = 0$. Otherwise, let $r$ be a row in $(\nu, J)^{(a)}$ of minimal length $\ell$ with rigging $x$. Then $e_a(\nu, J)$ is the rigged configuration which removes a box from row $r$, sets the new rigging of $r$ to be $x+1$, and changes all other riggings such that the coriggings remain fixed.

\item[\defn{$f_a$}:] Let $r$ be a row in $(\nu, J)^{(a)}$ of maximal length $\ell$ with rigging $x$. Then $f_a(\nu, J)$ is the rigged configuration which adds a box to row $r$, sets the new rigging of $r$ to be $x-1$, and changes all other riggings such that the coriggings remain fixed.
\end{itemize}
\end{dfn}

We define the remainder of the crystal structure on $\RC(\infty)$ by
\begin{gather*}
\varepsilon_a(\nu, J) = \max \{ k \in \ZZ \mid e_a^k(\nu, J) \neq 0 \}, \hspace{20pt} \varphi_a(\nu, J) = \inner{h_a}{\wt(\nu,J)} + \varepsilon_a(\nu, J),
\\ \wt(\nu, J) = -\sum_{a \in I} \lvert \nu^{(a)} \rvert \alpha_a.
\end{gather*}
From this structure, we have $p_\infty^{(a)} = \inner{h_a}{\wt(\nu,J)}$ for all $a\in I$.

\begin{thm}[{\cite{SalS15, SalS15II}}]
\label{thm:binf_isomorphism}
Let $\g$ be of symmetrizable type. Then $\RC(\infty) \iso B(\infty)$ as $U_q(\g)$-crystals.
\end{thm}

\begin{prop}[{\cite{SalS15, S06}}]
\label{prop:ep_phi}
Let $(\nu, J) \in \RC(\infty)$ and fix some $a \in I$. Let $x$ denote the smallest label in $(\nu,J)^{(a)}$. Then we have
\[
\varepsilon_a(\nu, J) = -\min(0, x) \hspace{40pt} \varphi_a(\nu, J) = p_{\infty}^{(a)} - \min(0, x).
\]
\end{prop}

It is a straightforward computation from the vacancy numbers to show that
\begin{equation}
\label{eq:convexity_exact}
\inner{h_a}{\lambda} - \sum_{b \in I} A_{ab} m_i^{(b)} = -p_{i-1}^{(a)} + 2 p_i^{(a)} - p_{i+1}^{(a)}.
\end{equation}
From this, we obtain the well-known convexity properties of the vacancy numbers.

\begin{lemma}[Convexity]
\label{lemma:convexity}
If $m_i^{(a)} = 0$, then we have
\[
2 p_i^{(a)} \geq p_{i-1}^{(a)} + p_{i+1}^{(a)}.
\]
Moreover, $p_{i-1}^{(a)} \geq p_i^{(a)} \leq p_{i+1}^{(a)}$ if and only if $p_{i-1}^{(a)} = p_i^{(a)} = p_{i+1}^{(a)}$.
\end{lemma}

In the sequel, we will refer to this lemma simply as convexity as we will frequently use it.

\section{Star-crystal structure}
\label{sec:*crystal}

\begin{dfn}
\label{def:RC_star_crystal_ops}
Fix some $a \in I$, and let $x$ be the smallest \emph{co}rigging in $(\nu,J)^{(a)}$.
\begin{itemize}
\item[\defn{$e_a^*$}:] If $x =0$, then $e_a(\nu, J) = 0$. Otherwise let $r$ be a row in $(\nu, J)^{(a)}$ of minimal length $\ell$ with corigging $x$. Then $e_a(\nu, J)$ is the rigged configuration which removes a box from row $r$ and sets the new corigging of $r$ to be $x+1$.

\item[\defn{$f_a^*$}:] Let $r$ be a row in $(\nu, J)^{(a)}$ of maximal length $\ell$ with corigging $x$. Then $f_a(\nu, J)$ is the rigged configuration which adds a box to row $r$ and sets the new colabel of $r$ to be $x-1$.
\end{itemize}
\end{dfn}

If $e_a^*$ removes a box from a row of length $\ell$ in $(\nu, J)$, then the the vacancy numbers change by the formula
\begin{equation}
\label{eq:change_vac_e}
\widetilde{p}_i^{(b)} = \begin{cases}
p_i^{(b)} & \text{if } i \leq \ell, \\
p_i^{(b)} + A_{ab} & \text{if } i > \ell.
\end{cases}
\end{equation}
On the other hand, if $f_a^*$ adds a box to a row of length $\ell$, then the vacancy numbers change by
\begin{equation}
\label{eq:change_vac_f}
\widetilde{p}_i^{(b)} = \begin{cases}
p_i^{(b)} & \text{if } i < \ell, \\
p_i^{(b)} - A_{ab} & \text{if } i \geq \ell.
\end{cases}
\end{equation}
Similar equations hold for $e_a$ and $f_a$ respectively. So the riggings of unchanged rows are changed according to Equation~\eqref{eq:change_vac_e} and Equation~\eqref{eq:change_vac_f} under $e_a$ and $f_a$, respectively. 

\begin{remark}
By Equation~\eqref{eq:change_vac_e} and Equation~\eqref{eq:change_vac_f}, the crystal operators $e_a$ and $f_a$ preserve all colabels of $(\nu, J)$ other than the row changed in $(\nu, J)^{(a)}$.
\end{remark}

\begin{ex}
\label{ex:running}
Consider type $D_4$ with Dynkin diagram
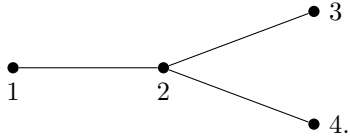
\[
\begin{tikzpicture}[xscale=2,yscale=.75] 
\node[circle,fill,scale=.45,label={below:$1$}] (1) at (0,0) {};
\node[circle,fill,scale=.45,label={below:$2$}] (2) at (1,0) {};
\node[circle,fill,scale=.45,label={right:$3$}] (3) at (2,1) {};
\node[circle,fill,scale=.45,label={right:$4.$}] (4) at (2,-1) {};
\path[-]
 (1) edge (2)
 (2) edge (3)
 (2) edge (4);
\end{tikzpicture}
\]
Let $(\nu,J)$ be the rigged configuration
\begin{align*}
(\nu, J) 
&= f_2^*f_3^* f_1^* f_2^* f_2^* f_4^* f_3^* f_1^* f_2^* (\nu_{\emptyset}, J_{\emptyset}) \\
&=
\begin{tikzpicture}[scale=.35,anchor=top,baseline=-18]
 \rpp{2}{0}{-1}
 \begin{scope}[xshift=6cm]
 \rpp{3,1}{-2,-1}{-3,-1}
 \end{scope}
 \begin{scope}[xshift=14cm]
 \rpp{2}{0}{-1}
 \end{scope}
 \begin{scope}[xshift=20cm]
 \rpp{1}{0}{0}
 \end{scope}
\end{tikzpicture}.
\end{align*}
Then
\[
f_2^*(\nu, J) = \begin{tikzpicture}[scale=.35,anchor=top,baseline=-18]
 \rpp{2}{0}{-1}
 \begin{scope}[xshift=6cm]
 \rpp{4,1}{-3,-1}{-5,-1}
 \end{scope}
 \begin{scope}[xshift=14cm]
 \rpp{2}{0}{-1}
 \end{scope}
 \begin{scope}[xshift=20cm]
 \rpp{1}{0}{0}
 \end{scope}
\end{tikzpicture}.
\]
\end{ex}

Let $\RC(\infty)^*$ denote the closure of $(\nu_{\emptyset}, J_{\emptyset})$ under $f_a^*$ and $e_a^*$. We define the remaining crystal structure by
\begin{gather*}
\varepsilon_a^*(\nu, J) = \max \{ k \in \ZZ \mid (e_a^*)^k(\nu, J) \neq 0 \}, 
\hspace{20pt} 
\varphi_a^*(\nu, J) = \inner{h_a}{\wt(\nu,J)} + \varepsilon_a^*(\nu, J),
\\ \wt(\nu, J) = -\sum_{a \in I} \lvert \nu^{(a)} \rvert \alpha_a.
\end{gather*}

\begin{remark}
\label{remark:duality}
We will say an argument holds by duality when we can interchange:
\begin{itemize}
\item ``label'' and ``colabel'';
\item $e_a$ and $e_a^*$;
\item $f_a$ and $f_a^*$.
\end{itemize}
For an example, compare the proof of Proposition~\ref{prop:ep_phi_star} with~\cite[Thm.~3.8]{Sakamoto14}.
\end{remark}

\begin{lemma}
The tuple $(\RC(\infty)^*, e_a^*, f_a^*, \varepsilon_a^*, \varphi_a^*, \wt)$ is an abstract $U_q(\g)$-crystal.
\end{lemma}

\begin{proof}
The proof that $(\RC(\infty)^*, e_a^*, f_a^*, \varepsilon_a^*, \varphi_a^*, \wt)$ is an abstract $U_q(\g)$-crystal is dual to that $\RC(\infty)$ is an abstract $U_q(\g)$-crystal under $e_a$ and $f_a$ in~\cite[Lemma~3.3]{SalS15}.
\end{proof}

\begin{prop}
\label{prop:ep_phi_star}
Let $(\nu, J) \in \RC(\infty)$ and fix some $a \in I$. Let $x$ denote the smallest colabel in $(\nu,J)^{(a)}$. Then we have
\[
\varepsilon_a^*(\nu, J) = -\min(0, x), \hspace{40pt} \varphi_a^*(\nu, J) = p_{\infty}^{(a)} - \min(0, x).
\]
\end{prop}

\begin{proof}
The following argument for $\varepsilon_a^*$ is essentially the dual to that given in~\cite[Thm.~3.8]{Sakamoto14}. We include it here as an example of Remark~\ref{remark:duality}.

It is sufficient to prove $\varepsilon_a^*(\nu, J) = -\max(0, x)$ since $p_{\infty}^{(a)} = \inner{h_a}{\wt(\nu, J)}$. If $x \geq 0 = \varepsilon_a^*(\nu, J)$, then $e_a^*(\nu, J) = 0$ by definition. Thus we proceed by induction on $\varepsilon_a^*(\nu,J)$ and assume $x < 0$. Let $(\nu',J') = e_a^*(\nu, J)$ and $y'$ denote the resulting colabel from a colabel $y$. In particular, we have $x' = x + 1$ and all other colabels follow Equation~\eqref{eq:change_vac_e}. Next, let $y$ denote the colabel of a row of length $j$. For $j  < \ell$, we have $y > x$ (equivalently $y \geq x - 1$) because we chose $\ell$ as large as possible. Thus $y' = y$, and hence $y' = y \geq x + 1 = x'$. For $j \geq \ell$, we have $y \geq x$ by the minimality of $x$ and $y' = y + 2$. Hence, $y' = y + 2 \geq x + 1 = x'$, and so $\varepsilon_a^*(\nu', J') = \varepsilon_a^*(\nu, J) - 1$ as desired.
\end{proof}

The rest of this section will amount to showing that Conditions~(\ref{item:star1})--(\ref{item:star6}) of Proposition~\ref{prop:weaker_conditions} hold.
Note that using Proposition \ref{prop:ep_phi_star} and \cite[Prop. 4.2]{SalS15}, we can rewrite Equation \eqref{eq:jump} as 
\begin{equation}
\label{eq:RCjump}
\begin{aligned}
\kappa_a(\nu,J) &= -\min(0,x_\ell) - \min(0,x_c) + \langle h_a , \wt(\nu,J) \rangle, \\
&= -\min(0,x_\ell) - \min(0,x_c) + p_\infty^{(a)},
\end{aligned}
\end{equation}
where $x_\ell$ and $x_c$ are the smallest label and colabel, respectively, in $(\nu,J)^{(a)}$.

\begin{lemma}
\label{lemma:kappa0}
Fix $(\nu, J) \in \RC(\infty)$ and $a \in I$. Assume $\kappa_a(\nu, J) = 0$. Then $f_a(\nu, J) = f_a^*(\nu, J)$.
\end{lemma}

\begin{proof}
Suppose that $f_a$ adds a box to a row of length $i$ with rigging $x$. Recall that $x = -\varepsilon_a(\nu, J)$. Suppose the longest row of $\nu^{(a)}$ has length $\ell > i$ and let $x_\ell$ denote any rigging of the longest row. Therefore, we have $x_\ell > x$ by the definition of $f_a$, and we have $p_\ell^{(a)} \leq p_{\infty}^{(a)}$ by convexity. Thus from the definition of $\varepsilon_a^*$, we have
\begin{equation}
\label{eq:contradiction_kappa0}
\varepsilon_a^*(\nu, J) \geq x_\ell - p_\ell^{(a)} > x - p_{\infty}^{(a)} = -\varepsilon_a(\nu, J) - p_{\infty}^{(a)}.
\end{equation}
This implies $\varepsilon_a^*(\nu, J) + \varepsilon_a(\nu, J) + p_{\infty}^{(a)} > 0$, which is a contradiction. Therefore $f_a$ must add a box to one of the longest rows of $\nu^{(a)}$. Moreover, if $p_\ell^{(a)} < p_{\infty}^{(a)}$, then Equation~\eqref{eq:contradiction_kappa0} would still hold and result in a contradiction. Similar statements holds for $f_a^*$ by duality.

Therefore $f_a$ and $f_a^*$ act on the longest row of $\nu^{(a)}$ and $p_i^{(a)} = p_{i+1}^{(a)} = p_{\infty}^{(a)}$. Let $x$ and $x^*$ denote the label of the row on which $f_a$ and $f_a^*$ act, respectively. Both of these labels decrease by $1$ after applying $f_a$ and $f_a^*$, respectively, by Equation~\eqref{eq:colabel_change} and Equation~\eqref{eq:label_change}, respectively. So it is sufficient to show $x = x^*$. Note that $x \leq x^*$ as the smallest colabel is the one with the largest rigging. Suppose $x < x^*$, then we have
\[
\varepsilon_a^*(\nu, J) \geq x^* - p_i^{(a)} > x - p_{\infty}^{(a)} \geq -\varepsilon_a(\nu, J) - p_{\infty}^{(a)},
\]
which is a contradiction. Therefore we have $f_a = f_a^*$.
\end{proof}

\begin{ex}
\label{ex:kappa0}
Let $(\nu,J)$ be the rigged configuration of type $D_4$ from Example \ref{ex:running}.  Then $\kappa_2(\nu,J) = 0$, and 
and 
\[
f_2(\nu,J) = 
\begin{tikzpicture}[scale=.35,anchor=top,baseline=-18]
 \rpp{2}{0}{-1}
 \begin{scope}[xshift=6cm]
 \rpp{4,1}{-3,-1}{-5,-1}
 \end{scope}
 \begin{scope}[xshift=14cm]
 \rpp{2}{0}{-1}
 \end{scope}
 \begin{scope}[xshift=20cm]
 \rpp{1}{0}{0}
 \end{scope}
\end{tikzpicture}.
\]
One can check that this agrees with $f_2^*(\nu,J)$ from Example \ref{ex:running}.
\end{ex}

\begin{lemma}
\label{lemma:kappa1}
Fix $(\nu, J) \in \RC(\infty)$ and $a \in I$. Assume $\kappa_a(\nu, J) \geq 1$. Then
\[
\varepsilon_a^*\bigl(f_a(\nu, J)\bigr) = \varepsilon_a^*(\nu, J),
\hspace{40pt}
\varepsilon_a(f_a^*(\nu,J)) = \varepsilon_a(\nu,J).
\]
\end{lemma}

\begin{proof}
Let $x^c$ denote the smallest colabel of $(\nu,J)^{(a)}$.
Let $x$ and $i$ denote the rigging and length of the row on which $f_a$ acts.  By the minimality of $x^c$, we have $x_c := p_i^{(a)} - x \geq x^c$. Note that the colabel of the row after the application of $f_a$ becomes
\begin{equation}
\label{eq:new_colabel}
\widetilde{x}_c := p_{i+1}^{(a)} - 2 - (x - 1) = p_{i+1}^{(a)} - x - 1,
\end{equation}
which implies that
\begin{equation}
\label{eq:colabel_change}
\widetilde{x}_c = x_c + p_{i+1}^{(a)} - p_i^{(a)} - 1.
\end{equation}

The remainder of the proof will be split into two cases: $x^c = p_i^{(a)} - x$ and $x^c < p_i^{(a)} - x$.

\case{$x^c = p_i^{(a)} - x$}

First, consider the case $p_{i+1}^{(a)} \leq p_i^{(a)}$. We also assume there exists a row of length $\ell > i$ of $\nu^{(a)}$, and let $x_\ell$ denote the rigging of that row. Thus $x < x_\ell$ and $x^c \leq p_\ell^{(a)} - x_\ell$ by the definition of $f_a$ and the minimality of $x^c$. Hence 
\[
p_i^{(a)} - x = x^c \leq p_\ell^{(a)} - x_\ell < p_\ell^{(a)} - x,
\]
which is equivalent to $p_i^{(a)} < p_\ell^{(a)}$. It must be the case, then, that $p_{i+1}^{(a)} > p_i^{(a)}$ by convexity, which is impossible. Thus the longest row of $\nu^{(a)}$ must be of length $i$. Convexity implies $p_i^{(a)} = p_{i+1}^{(a)} = p_\infty^{(a)}$, which results in
\begin{align*}
\kappa_a(\nu, J) & = p_{\infty}^{(a)} - \min(x, 0) - \min(x^c, 0)
\\ & = p_i^{(a)} - \min(x, 0) - \min\bigl(p_i^{(a)} - x, 0\bigr).
\end{align*}
Since $x$ was the rigging chosen by $f_a$, we must have $x \leq 0$. Additionally, if $p_i^{(a)} - x = x^c \leq 0$, we have
\[
1 \leq \kappa_a(\nu, J) = p_i^{(a)} - x - (p_i^{(a)} - x) = 0,
\]
which is a contradiction. Thus $x^c \geq 1$, which implies $\varepsilon_a^*(\nu, J) = 0 = \varepsilon_a^*\bigl(f_a(\nu, J))$ since $\widetilde{x}_c \geq 0$.

Next if $p_{i+1}^{(a)} = p_i^{(a)} + 1$, then $\varepsilon_a^*\bigl(f_a(\nu, J)\bigr) = \varepsilon_a^*(\nu, J)$ since $\widetilde{x}_c = x^c$, all other coriggings are fixed, and $\widetilde{x}_c = x_c$ by Equation~\eqref{eq:colabel_change} .

So we now assume $p_{i+1}^{(a)} \geq p_i^{(a)} + 2$, which implies $\widetilde{x}_c > x^c$. If there is another row with a corigging of $x^c$ or $x^c \geq 0$, then $\varepsilon_a^*\bigl(f_a(\nu, J)\bigr) = \varepsilon_a^*(\nu, J)$. So assume $f_a$ acts on the only row with a corigging of $x^c < 0$. Note that $p_i^{(a)} - x = x^c < 0$ and $x \leq 0$ implies $p_i^{(a)} < 0$.

We have $m_i^{(a)} = 1$ as, otherwise, we would either have a second corigging of $x^c$ or a smaller corigging from the minimality of $x$. Thus, by Equation~\eqref{eq:convexity_exact}, 
\begin{align*}
-2 - \sum_{b \neq a} A_{ab} m_i^{(b)} &= -p_{i-1}^{(a)} + 2p_i^{(a)} - p_{i+1}^{(a)} \\
&\leq -p_{i-1}^{(a)} + 2p_i^{(a)} - p_i^{(a)} - 2 \\
&= -p_{i-1}^{(a)} + p_i^{(a)} - 2.
\end{align*}
Since $m_i^{(b)} \geq 0$ and $-A_{ab} \geq 0$ for all $a \neq b$, we then have
\[
0 \leq -p_{i-1}^{(a)} + p_i^{(a)},
\]
or, equivalently, $p_{i-1}^{(a)} \leq p_i^{(a)}$. If $i$ is the length of the smallest row, then $0 \leq p_{i-1}^{(a)} \leq p_i^{(a)} < 0$ by convexity, which is a contradiction. Thus let $x_\ell$ denote rigging of the longest row in $\nu^{(a)}$ such that $\ell < i$, and by convexity, we have $p_\ell^{(a)} \leq p_i^{(a)}$. By the definition of $f_a$, we have $x_\ell \geq x$. Thus, we have $p_\ell^{(a)} - x_\ell \leq p_i^{(a)} - x$. However, by the unique minimality of $x^c$, we have $p_\ell^{(a)} - x_\ell > x^c = p_i^{(a)} - x$. This is a contradiction. Therefore $\varepsilon_a^*\bigl(f_a(\nu, J)\bigr) = \varepsilon_a^*(\nu, J)$.

\case{$x^c < p_i^{(a)} - x$}

Assume $\varepsilon_a^*\bigl(f_a(\nu,J)\bigr) \neq \varepsilon_a^*(\nu, J)$.  Then
\begin{equation}
\label{eq:difference_epsilon_change}
p_{i+1}^{(a)} - p_i^{(a)} - 1 < x^c+ x - p_i^{(a)} < 0,
\end{equation}
as, otherwise, the new corigging is not smaller than the minimal corigging (i.e., $\widetilde{x}_c < x^c$), which occurs on a different row and does not change under $f_a$. We rewrite Equation~\eqref{eq:difference_epsilon_change} as
\begin{equation}
\label{eq:rewrite_less_than}
p_{i+1}^{(a)} - 1 < x^c + x < p_i^{(a)}.
\end{equation}
Suppose there exists a row of length $\ell > i$ in $\nu^{(a)}$.  Then $x_\ell > x$ and $x^c \leq p_\ell^{(a)} - x_\ell$. Therefore,
\begin{equation}
\label{eq:double_inequality}
p_{i+1}^{(a)} - \ell < p_\ell^{(a)} - x_\ell + x < p_\ell^{(a)},
\end{equation}
which implies $p_{i+1}^{(a)} \leq p_j^{(a)}$ for all $\ell \geq j > i$ by convexity. Note that Equation~\eqref{eq:double_inequality} implies that $\ell > i+1$ since, otherwise, we would have $p_{i+1}^{(a)} < p_{i+1}^{(a)}$. We also have $p_{i+1}^{(a)} \leq p_j^{(a)}$ for all $j > i$ if there does not exist a row of length $\ell > i$. Since there does not exist a row of length $i+1$, we must have $p_i^{(a)} \leq p_{i+1}^{(a)} \leq p_{i+2}^{(a)}$ by convexity. Yet, Equation~\eqref{eq:rewrite_less_than} implies
\[
p_{i+1}^{(a)} < p_i^{(a)} \leq p_{i+1}^{(a)},
\]
but this is a contradiction. Therefore $\varepsilon_a^*\bigl(f_a(\nu, J)\bigr) = \varepsilon_a^*(\nu, J)$.
\end{proof}

\begin{ex}
\label{ex:kappa1}
%
%
%
%
%
%
%
%
%
%
%
%
%
%
%
Again, let $(\nu,J)$ be the rigged configuration of type $D_4$ from Example \ref{ex:running}.  Then $\varepsilon_3(\nu,J) = 0$, $\varepsilon_3^*(\nu,J) = 1$, and $\kappa_3(\nu,J) = 1$.  We have
\begin{align*}
f_3(\nu,J) &=
\begin{tikzpicture}[scale=.35,anchor=top,baseline=-18]
 \rpp{2}{0}{-1}
 \begin{scope}[xshift=6cm]
 \rpp{3,1}{-1,-1}{-1,-1}
 \end{scope}
 \begin{scope}[xshift=14cm]
 \rpp{3}{-1}{-2}
 \end{scope}
 \begin{scope}[xshift=20cm]
 \rpp{1}{0}{0}
 \end{scope}
\end{tikzpicture}\\
f_3^*(\nu,J) &= 
\begin{tikzpicture}[scale=.35,anchor=top,baseline=-18]
 \rpp{2}{0}{-1}
 \begin{scope}[xshift=6cm]
 \rpp{3,1}{-2,-1}{-2,-1}
 \end{scope}
 \begin{scope}[xshift=14cm]
 \rpp{3}{0}{-2}
 \end{scope}
 \begin{scope}[xshift=20cm]
 \rpp{1}{0}{0}
 \end{scope}
\end{tikzpicture}.
\end{align*}
Then $\varepsilon_3^*\bigl(f_3(\nu,J)\bigr) = 1$ and $\varepsilon_3\bigl(f_3^*(\nu,J)\bigr) = 0$.
\end{ex}

\begin{lemma}
\label{lemma:kappa2}
Fix $(\nu, J) \in \RC(\infty)$ and $a \in I$. Assume $\kappa_a(\nu, J) \geq 2$. Then
\[
f_af_a^*(\nu, J) = f_a^* f_a(\nu, J).
\]
\end{lemma}

\begin{proof}
Suppose $f_a$ (resp., $f_a^*$) acts on row $r$ of length $i$ (resp., row $r^*$ of length $i^*$) with rigging $x$ (resp., $x^*$). Without loss of generality, let $r$ (resp., $r^*$) be the northernmost such row in the diagram of $\nu^{(a)}$. Let $x_c^* = p_{i^*}^{(a)} - x^*$ and $x_c = p_i^{(a)} - x$. Note that $x \leq x^*$ and $x_c^* \leq x_c$. Applying $f_a^*$, the new rigging (and the only changed rigging) is
\begin{equation}
\label{eq:label_change}
\widetilde{x}^* = x^* + p_{i^*+1}^{(a)} - p_{i^*}^{(a)} - 1.
\end{equation}
Recall that Equation~\eqref{eq:colabel_change} gives the new corigging (and only changed corigging)
\[
\widetilde{x}_c = x_c + p_{i+1}^{(a)} - p_i^{(a)} - 1
\]
after applying $f_a$. We split the proof into three cases: the first two are cases in which $r \neq r^*$ and the last is when $r = r^*$.

\case{$f_a$ acts on row $r \neq r^*$ in $f_a^*(\nu, J)$}

Suppose $f_a f_a^*(\nu, J) \neq f_a^* f_a(\nu, J)$.  This is equivalent to $f_a^*$ acting on row $r' \neq r^*$ in $f_a(\nu, J)$. Note that we must have $r' = r$, since $f_a$ preserves all other colabels. From Equation~\eqref{eq:colabel_change}, we must have $p_{i+1}^{(a)} < p_i^{(a)} + 2$, as otherwise $\widetilde{x}_c > x_c \geq x_c^*$, which would imply $r = r' = r^*$ and be a contradiction. Next, consider when $p_{i+1}^{(a)} = p_i^{(a)} + 1$. Thus we have $\widetilde{x}_c = x_c$. Since $r' \neq r^*$, we must have $i = i^*$ and $x_c^* = \widetilde{x}_c = x_c$. However, this contradicts the assumption $r \neq r^*$ as we have $x = x^*$.

Hence $p_{i+1}^{(a)} \leq p_i^{(a)}$. Suppose $i$ is the length of the longest row of $\nu^{(a)}$.  Then $p_i^{(a)} = p_{i+1}^{(a)} = p_{\infty}^{(a)}$ by convexity. Moreover, we have $\widetilde{x}_c = x_c - 1 = x_c^*$ since $r, r' \neq r^*$. Note that since $f_a$ (resp., $f_a^*$) acts on $r$ (resp., $r^*$), we must have $x \leq 0$ (resp., $x_c^* \leq 0$). Therefore, we have
\begin{align*}
2 \leq \kappa_a(\nu, J) & = p_i^{(a)} - \min(x, 0) - \min(x_c^*, 0)
\\ & = p_i^{(a)} - x - x_c^*
\\ & = p_i^{(a)} - x - (p_i^{(a)} - x - 1) = 1,
\end{align*}
which is a contradiction.

Suppose there exists a row $r_\ell$ of length $\ell > i$ in $\nu^{(a)}$. Let $x_\ell$ denote the rigging of $r_\ell$, and note $x < x_\ell$ by our assumption. Therefore, we have
\[
p_{i+1}^{(a)} - x - 1  = \widetilde{x}_c \leq x_c^* \leq p_\ell^{(a)} - x_l < p_\ell^{(a)} - x,
\]
which implies $p_{i+1}^{(a)} \leq p_\ell^{(a)}$. Assume there exists a row of length $i+1$ in $\nu^{(a)}$ with rigging $x_{i+1}$.  It follows that
\[
p_{i+1}^{(a)} - x > p_{i+1}^{(a)} - x_{i+1} \geq x_c^* = p_{i^*}^{(a)} - x^*,
\]
which is equivalent to
\[
p_{i+1}^{(a)} - p_{i^*}^{(a)} > x - x^*.
\]
Furthermore,
\[
p_{i+1}^{(a)} - x - 1 = \widetilde{x}_c \leq x_c^* = p_{i^*}^{(a)} - x^*,
\]
which results in
\begin{equation}
\label{eq:changing_f_star_inequality}
x - x^* \geq p_{i+1}^{(a)} - p_{i^*}^{(a)} - 1.
\end{equation}
Additionally, Equation~\eqref{eq:changing_f_star_inequality} is necessarily a strict inequality if $i^* > i$ because it must be the case that $\widetilde{x}_c < x_c^*$.
Hence
\[
p_{i+1}^{(a)} - p_{i^*}^{(a)} > x - x^* \geq p_{i+1}^{(a)} - p_{i^*}^{(a)} - 1,
\]
which is a contradiction for $i^* > i$ as the right inequality becomes a strict inequality.

Next, note that $\widetilde{x}^* = x^* + p_{i^*+1}^{(a)} - p_{i^*}^{(a)} - 1 \geq x$ since $f_a$ acts on $r$. Hence 
\begin{equation}
\label{eq:unchanged_f_inequality}
p_{i^*+1}^{(a)} - p_{i^*}^{(a)} - 1 \geq x - x^*,
\end{equation}
which is a strict inequality for $i \leq i^*$. Thus 
\[
p_{i^*+1}^{(a)} - p_{i^*}^{(a)} - 1 \geq x - x^* \geq p_{i+1}^{(a)} - p_{i^*}^{(a)} - 1,
\]
or, equivalently, $p_{i^*+1}^{(a)} \geq p_{i+1}^{(a)}$. Since $i \leq i^*$, we have $p_{i^*+1}^{(a)} > p_{i+1}^{(a)}$, and hence $i = i^*$ cannot occur.

Now suppose $i^* < i$. Therefore $x_c < p_{i+1}^{(a)} - x_{i+1}$, which implies
\[
p_{i+1}^{(a)} - x - 1  = \widetilde{x}_c \leq x_c^* < p_{i+1}^{(a)} - x_{i+1} < p_{i+1}^{(a)} - x.
\]
So $p_{i+1}^{(a)} < p_{i+1}^{(a)}$, which is a contradiction.

Finally, if there does not exist a row of length $i+1$, then $p_i^{(a)} = p_{i+1}^{(a)} = p_\ell^{(a)}$ by convexity, so the argument given above will still yield a contradiction. Hence, $f_a f_a^*(\nu,J) = f_a^* f_a(\nu, J)$.

\case{$f_a$ acts on row $r' \neq r$ in $f_a^*(\nu, J)$}

Note that $r' = r^*$, where $r^* \neq r$, as $f_a^*$ fixes all other riggings. So from Lemma~\ref{lemma:kappa1}, we have
\[
\widetilde{x}^* = -\varepsilon_a\bigl(f_a^*(\nu, J)\bigr) = -\varepsilon_a(\nu, J) = x,
\]
and hence $i \leq i^*$. Therefore $x < x^*$ as $x = x^*$ implies $r = r^*$. Thus,
\[
x  = \widetilde{x}^* = x^* + p_{i^*+1}^{(a)} - p_{i^*}^{(a)} - 1 > x + p_{i^*+1}^{(a)} - p_{i^*}^{(a)} - 1,
\]
and hence $p_{i^*+1}^{(a)} \leq p_{i^*}^{(a)}$. Dually, $f_a^*$ acts on row $r$ in $f_a(\nu, J)$ since this would contradict $f_a$ acting on $r^* \neq r$ from the previous case. Similarly, 
\[
\widetilde{x}_c = -\varepsilon_a^*\bigl(f_a^*(\nu, J)\bigr) = -\varepsilon_a^*(\nu, J) = x_c^*
\]
by the dual version of Lemma~\ref{lemma:kappa1}, implying $i^* \leq i$. Hence, $i = i^*$ and
\[
p_i^{(a)} - x^* = x_c^* = \widetilde{x}_c = x_c + p_{i+1}^{(a)} - p_i^{(a)} - 1 = - x +  p_{i+1}^{(a)} - 1,
\]
which yields $x - x^* = p_{i+1}^{(a)} - p_i^{(a)} - 1$. Thus
\begin{align*}
x & = \widetilde{x}^* = x^* + p_{i+1}^{(a)} - p_i^{(a)} - 1
\\ p_i^{(a)} - x^* = x_c^* & = \widetilde{x}_c = x_c + p_{i+1}^{(a)} - p_i^{(a)} - 1 = p_{i+1}^{(a)} - x - 1 \leq p_i^{(a)} - x - 1,
\end{align*}
which implies $x - x^* \leq -1$. Hence,
\[
p_i^{(a)} - p_{i+1}^{(a)} = x^* - x - 1 \leq -2,
\]
and this contradicts $0 \leq p_i^{(a)} - p_{i+1}^{(a)}$. 


\case{$r = r^*$}

From $x = x^*$ and $i = i^*$, we have 
\begin{align*}
\widetilde{x}^* & = x^* + p_{i+1}^{(a)} - p_i^{(a)} - 1 = x + p_{i+1}^{(a)} - p_i^{(a)} - 1,
\\ \widetilde{x}_c & = x_c + p_{i+1}^{(a)} - p_i^{(a)} - 1 = p_i^{(a)} - x  + p_{i+1}^{(a)} - p_i^{(a)} - 1 = p_{i+1}^{(a)} - x - 1,
\\ x_c^* & = x_c = p_i^{(a)} - x.
\end{align*}

If $p_{i+1}^{(a)} \leq p_i^{(a)} + 1$, then $\widetilde{x}^* \leq x$ and $\widetilde{x}_c \leq x_c^*$. Hence, $f_a$ and $f_a^*$ select row $r$ in $f_a^*(\nu, J)$ and $f_a(\nu, J)$, respectively, and so we have $f_a f_a^*(\nu, J) = f_a^* f_a(\nu, J)$. Next, consider the case when $p_{i+1}^{(a)} \geq p_i^{(a)} + 2$.  Then $\widetilde{x}^* > x$ and $\widetilde{x}_c > x_c^*$. If $m_i^{(a)} \geq 2$, then there exists a row $r' \neq r$ such that $x^* = x$ and $x_c^* = x_c$. So $f_a$ and $f_a^*$ select row $r'$ in $f_a^*(\nu, J)$ and $f_a(\nu, J)$, respectively, and thus we have $f_a f_a^*(\nu, J) = f_a^* f_a(\nu, J)$.
If $m_i^{(a)} = 1$, then, as in Lemma~\ref{lemma:kappa1}, we have
\[
-2 - \sum_{b \neq a} A_{ab} m_i^{(b)} = -p_{i-1}^{(a)} + 2p_i^{(a)} - p_{i+1}^{(a)} \leq -p_{i-1}^{(a)} + p_i^{(a)} - 2.
\]
from Equation~\eqref{eq:convexity_exact}.  This implies $p_{i-1}^{(a)} \leq p_i^{(a)}$.  We consider the case when $r$ is the smallest row of $\nu^{(a)}$, which implies $r$ is the unique row with rigging $x$ and corigging $x_c^*$. Moreover, we have $0 \leq p_{i-1}^{(a)} \leq p_i^{(a)}$ by convexity. Because we are acting on $r$ by $f_a$ and $f_a^*$, we have $x \leq 0$ and $x_c^* \leq 0$. Hence,
\[
0 \leq p_i^{(a)} \leq p_i^{(a)} - x = x_c^* \leq 0 \implies 
0 = p_i^{(a)} = x = x_c^*. 
\]
Therefore $f_a$ (resp., $f_a^*$) acts on a row of length $0$ in $f_a^*(\nu, J)$ (resp., $f_a(\nu, J)$) as all other riggings (resp., coriggings) are positive. Moreover, the resulting rigging is $-1$ in both cases, and so $f_a f_a^*(\nu, J) = f_a^* f_a(\nu, J)$.

Now assume there exists a row $r_\ell$ of length $\ell < i$ in $\nu^{(a)}$, and without loss of generality, suppose $\ell$ is maximal. Let $x_\ell$ denote the rigging of $r_\ell$, and by the definition of $f_a$ and $f_a^*$, we have $x_\ell \geq x$ and $p_\ell^{(a)} - x_\ell \geq x_c^* = p_i^{(a)} - x$. By convexity, we have $p_\ell^{(a)} \leq p_{i-1}^{(a)} \leq p_i^{(a)}$. Therefore, we have
\[
p_i^{(a)} - x \leq p_\ell^{(a)} - x_\ell \leq p_i^{(a)} - x,
\]
and so $p_i^{(a)} - x = p_\ell^{(a)} - x_\ell$. Moreover, if $p_\ell^{(a)} < p_i^{(a)}$, then we have $x_\ell < x$, which cannot occur, and hence we also have $x = x_\ell$. Therefore $f_a$ and $f_a^*$ acts on $r_\ell$ in $f_a^*(\nu, J)$ and $f_a(\nu, J)$, respectively. Thus we have $f_a f_a^*(\nu, J) = f_a^*f_a(\nu, J)$.
\end{proof}

\begin{ex}
\label{ex:kappa2}
Continuing our running example, let $(\nu,J)$ be the rigged configuration of type $D_4$ from Example \ref{ex:running}.  Then $\kappa_4(\nu,J) = 2$ and 
\[
f_4^*f_4(\nu,J) = f_4f_4^*(\nu,J) = 
\begin{tikzpicture}[scale=.35,anchor=top,baseline=-18]
 \rpp{2}{0}{-1}
 \begin{scope}[xshift=5cm]
 \rpp{3,1}{-1,-1}{-1,-1}
 \end{scope}
 \begin{scope}[xshift=11.5cm]
 \rpp{2}{0}{-1}
 \end{scope}
 \begin{scope}[xshift=16.5cm]
 \rpp{3}{-1}{-2}
 \end{scope}
\end{tikzpicture}.
\]
With $a=4$, we have the diagram
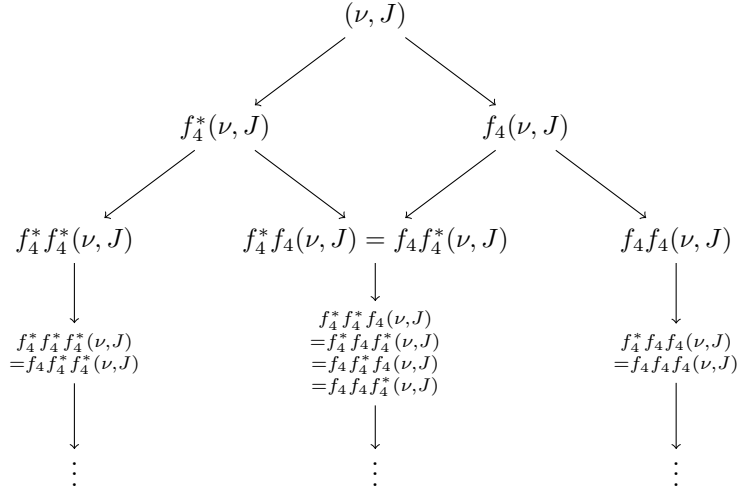
\[
\begin{tikzpicture}[xscale=2,yscale=1.5,font=\normalsize]
\node (t) at (0,0) {$(\nu,J)$};
\node (s) at (-1,-1) {$f_4^*(\nu,J)$};
\node (a) at (1,-1) {$f_4(\nu,J)$};
\node (ss) at (-2,-2) {$f_4^*f_4^*(\nu,J)$};
\node (aa) at (2,-2) {$f_4f_4(\nu,J)$};
\node (as) at (0,-2) {$f_4^*f_4(\nu,J) = f_4f_4^*(\nu,J)$};
\node (sss) at (-2,-3) {$\substack{f_4^*f_4^*f_4^*(\nu,J) \\ = f_4f_4^*f_4^*(\nu,J)}$};
\node (asa) at (0,-3) {$\substack{f_4^*f_4^*f_4(\nu,J) \\ = f_4^*f_4f_4^*(\nu,J) \\ = f_4f_4^*f_4(\nu,J) \\ = f_4f_4f_4^*(\nu,J)}$};
\node (aaa) at (2,-3) {$\substack{f_4^*f_4f_4(\nu,J) \\ = f_4f_4f_4(\nu,J)}$};
\foreach \x in {-2,0,2}
 {\node at (\x,-4) {$\vdots$};}
\path[->]
 (t) edge (s)
 (t) edge (a)
 (s) edge (ss)
 (s) edge (as)
 (a) edge (aa)
 (a) edge (as)
 (ss) edge (sss)
 (as) edge (asa)
 (aa) edge (aaa)
 (sss) edge (-2,-3.85)
 (asa) edge (0,-3.85)
 (aaa) edge (2,-3.85);
\end{tikzpicture}
\]
As discussed in \cite[Cor. 2.8]{CT15}, $\kappa_4(\nu,J)$ counts how many times one must apply either $f_4$ or $f_4^*$ to $(\nu,J)$ to reach a point where $f_4$ and $f_4^*$ have the same affect.
\end{ex}


\begin{thm}
Let $e_a$ and $f_a$ be the crystal operators given by Definition~\ref{def:RC_crystal_ops}, and let $e_a^*$ and $f_a^*$ be given by Definition~\ref{def:RC_star_crystal_ops}. Then we have
\[
e_a^*  = * \circ e_a \circ *, \ \ \ \ \ \
f_a^*  = * \circ f_a \circ *.
\]
\end{thm}

\begin{proof}
We show the conditions of Proposition~\ref{prop:star_properties} hold for $\RC(\infty)$ with the given crystal operations. Fix some $(\nu, J) \in \RC(\infty)$ and $a \in I$. 

We first note the fact that $f_a(\nu, J)$, $f_a^*(\nu, J) \neq 0$ follows immediately from the definitions. So we have Condition~(\ref{item:star1}). Now let $b \in I$.  As $f_b$ acts on labels and preserves colabels in $(\nu,J)^{(k)}$, for $k \neq b$ in $I$, and $f_a^*$ acts on colabels and preserves labels in $(\nu,J)^{(k)}$, for $k \neq a$ in $I$, it follows that $f_a^* f_b (\nu, J) = f_b f_a^* (\nu, J)$ for all $a \neq b$. Hence Condition~(\ref{item:star2}) is satisfied.

Lemma~\ref{lemma:kappa0} implies Condition~(\ref{item:star4}).

Lemma~\ref{lemma:kappa1} implies Condition~(\ref{item:star5}).

Lemma~\ref{lemma:kappa2} implies Condition~(\ref{item:star6}).

Thus it remains we prove Condition~(\ref{item:star3}), that $\kappa_a(\nu,J) \ge 0$. We prove this by induction on the depth of $(\nu,J)$.  Observe that $\kappa_a(\nu_\emptyset,J_\emptyset) = 0$, which is our base case.  Now suppose $\kappa_a(\nu,J) \ge 0$ for all $(\nu,J) \in \RC(\infty)$ at depth at most $d$.  It suffices to show that $\kappa_a\bigl(f_a(\nu,J)\bigr) \ge 0$ and $\kappa_a\bigl(f_a^*(\nu,J)\bigr) \ge 0$.

Note that all labels, except for the row of $\nu^{(a)}$ at which the box was added, possibly change by adding $-A_{ab}$ under $f_a$ by Equation~\eqref{eq:change_vac_f}. Additionally, $p_{\infty}^{(b)}$ changes by $-A_{ab}$. Thus for $b \neq a$ a label, and hence possibly $-\varepsilon_b(\nu, J)$, increases by $-A_{ab}$ and the colabels, and hence $-\varepsilon_b^*(\nu, J)$, stay fixed. Therefore, by the above and its dual, for $a \neq b$, we have
\[
\kappa_b\bigl(f_a(\nu, J)\bigr), \kappa_b\bigl(f_a^*(\nu, J)\bigr) \geq \kappa_b(\nu, J) \geq 0,
\]
since $-A_{ab} \geq 0$. Now it is sufficient to show that $\varepsilon_a^*(\nu, J)$ increases by at least $1 - \kappa_a(\nu, J)$ because $\varphi_a(\nu, J) = p_{\infty}^{(a)} + \varepsilon_a(\nu, J)$ decreases by $1$ after the application of $f_a$. If $\kappa_a(\nu, J) = 0$, then Lemma~\ref{lemma:kappa0} gives $f_a(\nu, J) = f_a^*(\nu, J)$, and so $\varepsilon_a^*(\nu, J)$ is increased by $1$. By Lemma~\ref{lemma:kappa1}, we have $\varepsilon_a^*(\nu, J) = \varepsilon_a^*\bigl(f_a(\nu, J)\bigr)$ when $\kappa_a(\nu, J) \geq 1$. Note that by our assumption, this is all possible values. Therefore, both $\kappa_a\bigl(f_a(\nu, J)\bigr), \kappa_a\bigl(f_a^*(\nu, J)\bigr) \geq 0$, as required.

Thus Conditions~(\ref{item:star1})--(\ref{item:star6}) are satisfied. Moreover, $\RC(\infty)$ and $\RC(\infty)^*$ are both generated by the highest weight element $(\nu_{\emptyset}, J_{\emptyset})$. Hence $\RC(\infty) = \RC(\infty)^*$ and the result follows from Proposition~\ref{prop:weaker_conditions}.
\end{proof}

Now from Definition~\ref{def:RC_crystal_ops} and Definition~\ref{def:RC_star_crystal_ops}, we have that the $*$-involution is given as follows.

\begin{cor}
\label{cor:RC_star_involution}
The $*$-involution on $\RC(\infty)$ is given by replacing every rigging $x$ of a row of length $i$ in $(\nu, J)^{(a)}$ by the corresponding corigging $p_i^{(a)} - x$ for all $(a, i) \in \HH$.
\end{cor}

Let $\g$ and $\widehat{\g}$ be symmetrizable Kac-Moody algebras such that there exists a folding of the Dynkin diagram of $\g$ to the Dynkin diagram of $\widehat{\g}$ with corresponding index sets $I$ and $\widehat{I}$, respectively.  Consider the map $\phi \colon \widehat{I} \searrow I$ induced by such a Dynkin diagram folding and consider a sequence $(\gamma_a \in \ZZ_{>0})_{a \in I}$ such that the map $\Psi \colon P \longrightarrow \widehat{P}$ given by
\[
\Lambda_a \mapsto \gamma_a \sum_{b \in \phi^{-1}(a)} \Lambda_b
\]
also satisfies
\[
\alpha_a \mapsto \gamma_a \sum_{b \in \phi^{-1}(a)} \alpha_b.
\]
This induces a \defn{virtualization map} $v$ of $B(\infty)$ of type $\g$ to that of type $\widehat{\g}$. In particular, on $\RC(\infty)$, the image $(\widehat{\nu}, \widehat{J})$ of a rigging configuration $(\nu, J)$ is given by
\[
\widehat{m}_{\gamma_a i}^{(b)} = m_i^{(a)},
\hspace{30pt}
\widehat{J}_{\gamma_a i}^{(b)} = \gamma_a J_i^{(a)},
\]
for all $b \in \phi^{-1}(a)$. We refer the reader to~\cite{OSS03III, SalS15III, SchillingS15} for more details.

\begin{cor}
Let $v$ be a virtualization map on $B(\infty)$ of type $\g$ to $\widehat{\g}$. Then
\[
\ast \circ v = v \circ \ast.
\]
\end{cor}

\begin{proof}
This follows from the fact
$
\widehat{p}_{\gamma_a i}^{(b)} = \gamma_a p_i^{(a)}
$
for all $b \in \phi^{-1}(a)$ and Corollary~\ref{cor:RC_star_involution}.
\end{proof}

\section{Highest weight crystals}
\label{sec:hw_crystals}

We wish to classify the subcrystal of $\RC(\infty)$ which is isomorphic to $B(\lambda)$ with respect to the $\ast$-crystal structure.  In particular, defining $B(\lambda)$ requires the additional condition that $\varphi_a^*(\nu, J) = \max\{k \in \ZZ \mid (f_a^*)^k(\nu, J) \neq 0\}$.  For example, the condition $\varphi_a(\nu,J) = \max\{ k \in \ZZ \mid f_a^k(\nu,J) \neq 0 \}$ means, for all riggings $x$ corresponding to a row of length $i$ in $\nu^{(a)}$, we have $x \leq p_i^{(a)}$.  If we consider the natural dual to this, we have $p_i^{(a)} - x \leq p_i^{(a)}$, or equivalently $x \geq 0$. We show this is the correct condition by proving the dual version of~\cite[Thm.~6.1]{SalS15}.

For any $\lambda \in P^+$, we define
\[
\RC(\lambda) := \{ (\nu, J) \in \RC(\infty) \mid \max J_i^{(a)} \leq p_i^{(a)}(\nu; \lambda) \text{ for all } (a, i) \in \HH \},
\]
where
\begin{equation}
\label{eq:vacancy_numbers}
p_i^{(a)}(\nu; \lambda) := \inner{h_a}{\lambda} - \sum_{b \in I} A_{ab} \sum_{j \in \ZZ_{>0}} \min(i,j) m_j^{(b)}.
\end{equation}
Note that Equation~\eqref{eq:vacancy_numbers} differs from Equation~\eqref{eq:vacancy} by
\[
p_i^{(a)}(\nu) + \inner{h_a}{\lambda} = p_i^{(a)}(\nu; \lambda).
\]
When there is no danger of confusion, we will simply write $p_i^{(a)} = p_i^{(a)}(\nu; \lambda)$.

We consider a crystal structure on $\RC(\lambda)$ as that inherited from $\RC(\infty)$ under the natural projection except with $\wt(\nu, J) = \lambda - \sum_{a \in I} \lvert \nu^{(a)} \rvert \alpha_a$.

\begin{thm}[{\cite{SalS15,SalS15II,S06}}]
We have
$
\RC(\lambda) \iso B(\lambda).
$
\end{thm}

Using the $\ast$-crystal structure, we easily obtain~\cite[Prop.~8.2]{K95}.  (We refer the reader to \cite{K95} or \cite{HK02} for an exposition on the tensor product of crystals.  Note that we are using the opposite, anti-Kashiwara, convention.  The precise definition in this setting may be found, for example, in \cite{SalS15}.)

\begin{prop}
Let $\lambda \in P^+$. Then we have
\[
\RC(\lambda) \cong \{ t_\lambda \otimes (\nu,J) \in T_\lambda \otimes \RC(\infty) \mid \varepsilon_a^*(\nu,J) \le \langle h_a , \lambda \rangle \text{ for all } a \in I\}.
\]
\end{prop}

\begin{proof}
Fix some $(\nu, J) \in \RC(\infty)$. Let $x$ be a rigging of a row of length $i$. We have
\[
\inner{h_a}{\lambda} \geq \varepsilon_a^*(\nu, J) = -\min(0, p_i^{(a)} - x)
\]
if and only if
\[
p_i^{(a)} + \inner{h_a}{\lambda} \geq x.
\]
Recall that the left-hand side is the vacancy numbers in $\RC(\lambda)$ by Equation~\eqref{eq:vacancy_numbers}, and so we have the defining relation for $\RC(\lambda)$.
\end{proof}

By letting $\pi_{\lambda} \colon B(\infty) \longrightarrow B(\lambda)$ be the natural projection, we can rephrase the last proposition as
\[
\tau(b^*) = \varepsilon(b), \qquad \qquad \varepsilon(b^*) = \tau(b),
\]
where $\varepsilon(b) = \sum_{a \in I} \varepsilon_a(b) \Lambda_a$ and $\tau(b) = \min \{ \lambda \mid \pi_{\lambda}(b) \in B(\lambda) \}$.   In~\cite{SalS15II}, $\tau$ was called the difference statistic and can be explicitly given on rigged configurations by
\[
\tau(\nu, J) = \sum_{a \in I} \min_{i \in \ZZ_{>0}} \{p_i^{(a)} - \max J_i^{(a)} \} \Lambda_a.
\]

Now we formalize the dual version of $\RC(\lambda)$.

\begin{dfn}
Let $\RC(\lambda)^*$ denote the closure of $(\nu_{\emptyset}, J_{\emptyset})$ under $e_a^*$ and the following modified $f_a^*$, both using the vacancy numbers given by Equation~\eqref{eq:vacancy_numbers} to determine the colabels. Consider $f_a^*$ as in Definition~\ref{def:RC_star_crystal_ops} except define $f_a^*(\nu, J) = 0$ if in the result, there exists a rigging $x < 0$.
\end{dfn}

Note that the condition that $x \geq 0$ is equivalent to $p_i^{(a)} - x \leq p_i^{(a)}$. Hence, by duality, the proof of~\cite[Lemma~3.6]{S06} holds, and we obtain the following.

\begin{lemma}
\label{lemma:lower_regular}
Let $(\nu, J) \in \RC(\lambda)^*$. Then
\[
\varphi_a^*(\nu, J) = \max\{ k \in \ZZ \mid (f_a^*)^k (\nu, J) \neq 0 \}
\]
for all $a \in I$.
\end{lemma}

Let $\RC_{\lambda}(\infty)^* = T_{\lambda} \otimes \RC(\infty)$ with the $\ast$-crystal structure. Let $C = \{c\}$ be the crystal given by
\[
\wt(c) = 0, \qquad
\varphi_a^*(c) = \varepsilon_a^*(c) = 0, \qquad
f_a^*c = e_a^*c = 0,
\]
for all $a \in I$. Nakashima \cite[Thm. 3.1]{N99} has shown that the connected component generated by $c \otimes t_{\lambda} \otimes u_{\infty}$ is isomorphic to $B(\lambda)$.

In~\cite{SalS15}, the map $\psi_{\lambda,\mu} \colon \RC(\lambda) \longrightarrow \RC(\mu)$, for $\lambda \leq \mu$ in $P^+ \sqcup \{\infty\}$, is the identity map on rigged configurations. This follows because $e_a$ and $f_a$ are determined by the riggings alone, not the vacancy numbers, and so preserving the labels is sufficient to show $\psi_{\lambda,\mu}$ commutes with the crystal operators. However, for the $\ast$-crystal structure, we need to preserve coriggings, and as such, we need to take into account the shift in vacancy numbers. Thus, define a map $\psi_{\lambda,\mu}^* \colon \RC(\lambda)^* \longrightarrow \RC(\mu)^*$ as the identity on the partitions but with new riggings
\[
x' = x + \langle h_a , \mu - \lambda\rangle,
\]
where we make the convention that $\inner{h_a}{\infty} = 0$. Note that $\psi_{\lambda,\mu}^*$ commutes with the crystal operators (however, it only becomes a crystal embedding after an appropriate tensor product is taken to shift weights).

With this modification, Proposition~\ref{prop:ep_phi_star}, and Lemma~\ref{lemma:lower_regular}, we have the dual argument of~\cite[Thm.~6.1]{SalS15}. 

\begin{thm}
Let $C_{\emptyset}^*$ denote the connected component of $C \otimes \RC_{\lambda}(\infty)^*$ generated by $c \otimes (\nu_{\emptyset}, J_{\emptyset})$. The map $\Psi \colon C_{\emptyset}^* \longrightarrow \RC(\lambda)^*$ given by 
\[
c \otimes (\nu_{\lambda}, J_{\lambda}) \mapsto (\psi_{\lambda,\infty}^*)^{-1}(\nu_{\lambda}, J_{\lambda})
\]
is a weight-preserving bijection which commutes with $e_a^*$ and $f_a^*$ for every $a\in I$.
\end{thm}

\begin{cor}
Let $\g$ be of symmetrizable type. Then $\RC(\lambda)^* \iso B(\lambda)$.
\end{cor}

Hence, we can now construct an explicit crystal isomorphism $\RC(\lambda)^* \iso \RC(\lambda)$ by passing through $\RC(\infty)$.

\begin{cor}
Let $\g$ be a symmetrizable Kac-Moody algebra and let $\lambda \in P^+$. Define $\Xi \colon \RC(\lambda) \longrightarrow \RC(\lambda)^*$ by $\Xi(\nu, J) = (\nu, J')$, where the resulting riggings are
\[
x' = x + \inner{h_a}{\lambda}.
\]
Then $\Xi$ is a crystal isomorphism.
\end{cor}

\begin{proof}
We have $\Xi = (\psi_{\lambda,\infty}^*)^{-1} \circ \psi_{\lambda,\infty}$.
\end{proof}

\appendix
\section{\textsc{SageMath} examples}

The crystal $\RC(\infty)$ has been implemented in \textsc{SageMath} \cite{combinat,sage} by the second author and the $*$-crystal has been implemented by the first author.

In order to make the rigged configurations display in a vertical-space-saving manner, we use the following.
\begin{lstlisting}
sage: RiggedConfigurations.global_options(display="horizontal")
\end{lstlisting}
Now let's check our running example using \textsc{SageMath}.  To initialize the rigged configuration from Example \ref{ex:running}, one does
\begin{lstlisting}
sage: RC = crystals.infinity.RiggedConfigurations("D4")
sage: RCstar = crystals.infinity.Star(RC)
sage: nu0star = RCstar.module_generators[0]
sage: nustar = nu0star.f_string([2,1,3,4,2,2,1,3,2]); nustar
-1[ ][ ]0   -3[ ][ ][ ]-2   -1[ ][ ]0   0[ ]0
            -1[ ]-1                          
sage: nustar.f(2)
-1[ ][ ]0   -5[ ][ ][ ][ ]-3   -1[ ][ ]0   0[ ]0
            -1[ ]-1
\end{lstlisting}
Continuing to Examples \ref{ex:kappa0}, \ref{ex:kappa1}, and \ref{ex:kappa2}, we must compute $\kappa_a(\nu,J)$ for $a\in I$.  
\begin{lstlisting}
sage: [nustar.jump(i) for i in nustar.index_set()]
[1, 0, 1, 2] 
\end{lstlisting}
To check Example \ref{ex:kappa0}, we must initialize the corresponding rigged configuration with respect to the usual crystal operators.
\begin{lstlisting}
sage: nu = RC(nustar.value)                            
sage: nu.f(2)
-1[ ][ ]0   -5[ ][ ][ ][ ]-3   -1[ ][ ]0   0[ ]0
            -1[ ]-1 
\end{lstlisting}
Since $\kappa_3(\nu,J) = 1$, we can verify Condition (\ref{item:star5}) of Proposition \ref{prop:star_properties} for this particular rigged configuration, which is the content of Example \ref{ex:kappa1}.
\begin{lstlisting}
sage: nustar.jump(3)
1
sage: RCstar(nu.f(3)).epsilon(3) == nustar.epsilon(3)
True
sage: RC(nustar.f(3).value).epsilon(3) == nu.epsilon(3)
True
\end{lstlisting}
Finally, Example \ref{ex:kappa2} gives a particular instance where $f_a$ and $f_a^*$ commute, since $\kappa_4(\nu,J) =2$.
\begin{lstlisting}
sage: nustar.jump(4)
2
sage: RCstar(nu.f(4)).f(4)
-1[ ][ ]0   -1[ ][ ][ ]-1   -1[ ][ ]0   -2[ ][ ][ ]-1
            -1[ ]-1                                  
sage: RC(nustar.f(4).value).f(4)
-1[ ][ ]0   -1[ ][ ][ ]-1   -1[ ][ ]0   -2[ ][ ][ ]-1
            -1[ ]-1   
\end{lstlisting}

\subsection*{Acknowledgements}

T.S.\ would like to thank Central Michigan University for its hospitality during his visit in October 2015, where this work originated.
This work was aided by computations in {\sc SageMath}~\cite{combinat,sage}.

\bibliography{RC_star}{}
\bibliographystyle{amsplain}
\end{document}